\documentclass[11pt,reqno,a4paper]{amsart}
\usepackage[latin2]{inputenc}
\usepackage{t1enc}
\usepackage{verbatim}
\usepackage{amsmath, amsthm, amssymb}
\usepackage{graphicx}
\usepackage{xcolor}
\usepackage{soul}

\theoremstyle{plain}
\newtheorem{theo}{Theorem}[section] 
\newtheorem{theorem}[theo]{Theorem}
\newtheorem{corollary}[theo]{Corollary}
\newtheorem{claim}{\sc Claim}

\newcommand{\rel}[1]{\mathbb{#1}}
\newcommand{\alg}[1]{\mathbf{#1}}
\newcommand{\var}[1]{\mathcal{#1}}

\newcommand{\A}{{\bf A}}

 \usepackage{mathptmx}  

\begin{document}

\title{Congruence permutability is prime}

\author{Gerg\H o Gyenizse}
\email{gergogyenizse@gmail.com}
\address {Bolyai Institute, University of Szeged, Hungary}
          
\author{Mikl\'os Mar\'oti}
\email{mmaroti@math.u-szeged.hu}
\address {Bolyai Institute, University of Szeged, Hungary}
           
\author{L\'aszl\'o Z\'adori}

\email{zadori@math.u-szeged.hu}
\address{Bolyai Institute, University of Szeged, R\'enyi Institute, Budapest, Hungary}

\thanks{The research of authors was partially supported by the grant NKFIH-K128042 of the Ministry for Innovation and Technology, Hungary, and by the Alfr\'ed R\'enyi Institute of Mathematics, ELKH, Hungary.}

\begin{abstract}
We give a combinatorial proof that congruence permutability is prime in the lattice of interpretability types of varieties. Thereby, we settle a 1984 conjecture of Garcia and Taylor.
\end{abstract}
\maketitle
\section{Introduction}

A {\em variety} or {\em equational class} is a class of all algebras (algebraic structures of a given signature) satisfying a given set of identities.
Let $\Gamma$ be a set of identities over a certain signature of a variety, and let $\mathcal{K}$ be a variety of some (possibly different) signature. We say that $\Gamma$  {\em interprets in the variety} $\mathcal{K}$ if by replacing the operation symbols in  $\Gamma$ by term expressions of $\mathcal{K}$---same symbols by same terms with arities kept---the so obtained set of identities holds in $\mathcal{K}$. A   {\em variety $\mathcal{K}_1$ interprets in a variety $\mathcal{K}_2$} if there is a set of identities $\Gamma$ that defines $\mathcal{K}_1$ and interprets in $\mathcal{K}_2$.

As easily seen, interpretability is a quasiorder on the class of varieties. The blocks of this quasiorder  are called the {\em interpretability types}. In \cite{GT} Garcia and Taylor introduced the {\em lattice of interpretability types of varieties} that is obtained by taking the quotient of the class of varieties quasiordered by  interpretability   and  the corresponding equivalence relation. The join in this lattice is described as follows. Let $\mathcal{K}_1$  and $\mathcal{K}_2$ be two varieties of disjoint signatures, defined by the sets $\Sigma_1$ and $\Sigma_2$ of identities, respectively.  Their {\em join} $\mathcal{K}_1\vee\mathcal{K}_2$ is the variety defined by $\Sigma_1\cup\Sigma_2$. The so defined join is compatible with the interpretability relation of varieties, and naturally yields the definition of the join operation in the lattice of interpretability types of varieties.

A {\em digraph} is a pair  $\rel G=(G; E(\rel G))$ where $G$ is a set and $E(\rel G)$ is a binary relation on $G$. Here $G$ is called the {\em vertex set} of $\rel G$ and $E(\rel G)$ the {\em edge relation} of $\rel G$. A {\em reflexive (symmetric, transitive)} digraph is a digraph whose edge relation is reflexive (symmetric, transitive).

A {\em compatible digraph of an algebra} $\alg A$ is a digraph whose vertex set coincides with the base set of $\alg A$ and whose edge relation is preserved by all of the basic operations of $\alg A$. A {\em compatible digraph in a variety} is  a compatible digraph of an algebra in the variety.

A {\em congruence} of an algebra $\A$ is an equivalence $\rho$ on $A$ such that $(A;\rho) $ is a compatible digraph of $\A$. An {\em algebra $\A$ is congruence permutable}, if  for any two congruences $\alpha$ and $\beta$ of $\A$, $\alpha\beta=\beta\alpha$.
A {\em variety is congruence permutable} if all of its members are congruence permutable. For example, the varieties of groups, rings and vector spaces are congruence permutable. The following characterization of congruence permutable varieties is due to Maltsev, cf. \cite{M}.

\begin{theorem}[Maltsev (1954)]\label{Maltsev}
Let $\mathcal{K}$ be a variety. Let $\alg F_2$ be the algebra freely generated by $x$ and $y$ in $\var {K}$. Then the following are equivalent.
\begin{enumerate}
\item $\var {K}$ is congruence permutable.
\item The reflexive digraph whose vertex set is $F_2$  and whose edge relation is the subalgebra generated by $\{(x,x),\ (x,y),\ (y,y) \}$ in $\alg F^2_2$ is symmetric.
\item The set $\{m(x,y,y)=x,\ m(y,y,x)=x\}$ of identities where $m$ is a ternary function symbol interprets in $\mathcal{K}$.
\end{enumerate}
\end {theorem}

A finite set of identities is called a {\em strong  Maltsev condition}. The set of identities occurring in the third condition of the preceding theorem is a typical example of a strong Maltsev condition. The interpretability types that contain the varieties in which a strong Maltsev condition interprets constitute a principal filter in the lattice of interpretability types of varieties. So the interpretability types containing the congruence permutable varieties form a principal filter in the lattice of interpretability types by Maltsev's theorem. We say that a strong Maltsev condition is {\em prime} if the related principal filter is a prime filter (equivalently, the smallest element of this principal filter is join-prime). In \cite{GT}, Garcia and Taylor formulated the conjecture that congruence permutability is prime. Formally, the strong Maltsev condition in item (3) of the above theorem  is prime.

In 1996, Tschantz announced a proof of this conjecture. However, his proof has remained unpublished, cf. \cite{tsc}. In the course of time, two partial results related to Garcia and Taylor's conjecture were published. In his PhD thesis \cite{S}, Sequeira proved that congruence permutability is prime with respect to varieties axiomatized by  identities  with  terms  of  depth  at  most  two. Kearnes and Tschantz gave a proof that congruence permutability is prime in the lattice of interpretability types of idempotent varieties, cf. Lemma 2.8  in \cite{KT}.   In the present paper, we give a proof of the primeness of congruence permutability in its full generality.

\section{A proof of the primeness of congruence permutability}

 A digraph is called a {\em complete digraph} if its edge relation is the full binary relation. In particular, complete digraphs are reflexive digraphs. The {\em complement of a digraph} $\rel G$ is the digraph whose vertex set is $G$ and whose edge set is $G^2\setminus E(\rel G)$. Let $\rel G$ be a digraph and $u$ a vertex of $\rel G$. Let $\rel G-u $ denote the digraph obtained by removing  the vertex $u$ and the edges incident with $u$ from $\rel G$. We call $u$ a {\em universal vertex} if for all $g\in G$, $u\leftrightarrow g$ in $\rel G$.

 Let $\rel H_i$, $i\in I$, be digraphs. Their product $\prod_{i\in I} \rel H_i$ is the digraph whose vertex set is $\prod_{i\in I} H_i$ and whose edge relation is 
 \[
 \{\ (h,h')\in (\textstyle\prod_{i\in I} H_i)^2:\ h(i)\to h'(i) \text{ in } \rel H_i \text{ for all } i\in I\ \}.
 \]
Let $\rel G$ and $\rel H$ be two digraphs. We define the digraph $\rel G^{\rel H}$ as follows. The vertex set of $\rel G^{\rel H}$ is the set of maps from the set $H$ to the set $G$. The edge relation of $\rel G^{\rel H}$ is defined by
$$ f\rightarrow f' \text{ if and only if } f(x)\rightarrow f'(y) \text{ in } \rel G \text{ for every }x\rightarrow y \text{ in } \rel H.$$ Observe that for any digraphs $\rel G,\ \rel H_1$ and $\rel H_2$, $\rel G^{\rel H_1\times \rel H_2}\cong \left(\rel G^{\rel H_1}\right)^{\rel H_2}$.

The heart of the proof of the main result of this paper is the following theorem on digraph powers.
\begin{theorem}
\label{two_isomorphic_powers}
For $1\leq i\leq 2$, let $\rel G_i$ be a digraph with a universal vertex $u_i$ and $\rel G^*_i$ the complement of the digraph $\rel G_i-u_i$. Let $\kappa$ be an infinite cardinal where $\kappa \geq \max(|G_1|,|G_2|)$, and  $\rel K$ a complete digraph of $\kappa$-many vertices. Then \[
\rel G_1^{\rel G^*_2\times\rel K}\cong\rel G_2^{\rel G^*_1\times\rel K}.
\]
\end{theorem}
\begin{proof} 
When $\rel G_1$ or $\rel G_2$ is a one-element complete graph, then both of the digraphs $\rel G_1^{\rel G^*_2\times\rel K}$ and $\rel G_2^{\rel G^*_1\times\rel K}$ are one-element complete graphs. So in what follows in this proof, we assume that $|G_1|,|G_2|\geq 2$. Let $G^*_i:=G_i\setminus\{u_i\}$ for $1\leq i\leq 2$.

For each vertex $f$ of $\rel G_1^{\rel G^*_2\times\rel K}$, let \[
C(f):=\{\ (g_1,g_2)\in G^*_1\times G^*_2:\ \text{ there exists } k\in K \text{ such that } f(g_2,k)=g_1\ \}.
\]

First, we give a characterization of the edge relation of $\rel G_1^{\rel G^*_2\times\rel K}$ by the use of the  sets $C(f)$.

\begin{claim}
Let $f,f'\in\rel G_1^{\rel G^*_2\times\rel K}$. The pair $(f,f')$ is not an edge of the digraph $\rel G_1^{\rel G^*_2\times\rel K}$ if and only if there exist $(g_1,g_2)\in C(f)$ and $(g'_1,g'_2)\in C(f')$ such that none of $(g_1,g'_1)$ and $(g_2,g'_2)$ are edges in $\rel G_1$ and $\rel G_2$, respectively. 
\end{claim}

 Clearly, $(f,f')$ is not an edge in $\rel G_1^{\rel G^*_2\times\rel K}$ if and only if there exist $g_1,\ g'_1\in G^*_1$, $g_2,\ g'_2\in G^*_2$ and $k,\ k'\in K$ such that $$f(g_2,k)=g_1,\  f(g'_2,k')=g'_1,$$ $(g_2,k)\to (g'_2,k')$ in $\rel G^*_2\times\rel K$ and $g_1\not\to g'_1$ in $\rel G_1$. By using the definitions of $C(f)$ and $C(f')$ and taking into account that $(g_2,k)\to (g'_2,k')$ in $\rel G^*_2\times\rel K$ is equivalent to $g_2\not\to g'_2$ in $\rel G_2$, we get the claim.

 We define the  equivalence $\varrho_1$ on $\rel G_1^{\rel G^*_2\times\rel K}$ by 
 $$(f, f')\in \varrho_1 \text{ if and only if }C(f)=C(f').$$
Clearly, the $\varrho_1$-block  that contains the constant $u_1$ map is a singleton. By the following claim, each of the other blocks has cardinality $2^\kappa$. Also, the subsets of $G^*_1\times G^*_2$ are in a bijective correspondence with the blocks of $\varrho_1$. 

\begin{claim}
For every non-empty subset $U$ of $G^*_1\times G^*_2$
$$B_U=\{\ f\in G_1^{ G^*_2\times K}:\ C(f)=U\ \}$$ is a $\rho_1$-block of cardinality $2^{\kappa}$.
\end{claim}
For any subset $S$ of $K$ with $|S|=|K|=\kappa$, let  $g_S$ be any surjective map in $U^S$. There is such a map, since $|U|\leq |G_1||G_2|\leq \kappa$. Now for every $S\subseteq K$ with $|S|=\kappa$, we define $f_S\in \rel G_1^{\rel G^*_2\times\rel K}$ by
\[
f_S(g_2,k):=\begin{cases}
g_1,\text{ if  $k\in S$ and $(g_1,g_2)=g_S(k)$,}\\
u_1\text{ otherwise.}
\end{cases}
\]
Now clearly, $f_S\in B_U$ for all $S\subseteq K$ with  $|S|=\kappa$.
Moreover, the number of the $f_S$ coincides with that of the subsets of cardinality $\kappa$ in $K$. So this number is $2^{\kappa}$, hence $2^{\kappa}\leq|B_U|$. On the other hand, $|B_U|\leq |G_1|^{|G_2||K|}=2^\kappa$. Since $B_U$ is non-empty, it is clearly a $\rho_1$-block. So the claim is proved.

We extend  the definition of $C$ onto $\rel G_2^{\rel G^*_1\times\rel K}$ by letting \[
C(h):=\{\ (g_1,g_2)\in G^*_1\times G^*_2 :\ \text{ there exists } k\in K\text{ such that } h(g_1,k)=g_2\ \}
\] for any $h\in G_2^{G^*_1\times K}$. We define the equivalence $\varrho_2$ on $\rel G_2^{\rel G^*_1\times\rel K}$ analogously to $\varrho_1$. With the notions defined in this paragraph, the analogues of Claim 1 and Claim 2 obviously hold.

By Claim 2 and its analogue, for every non-empty subset $U\subseteq G^*_1\times G^*_2$, the cardinalities of the $\rho_1$-block $B_U$ and the $\rho_2$-block $D_U$ that correspond to $U$ coincide with $2^{\kappa}$. Hence there exists a bijection from $B_U$ to $D_U$. If $U=\emptyset$, there is also a bijection from  the related $\rho_1$-block to the related $\rho_2$-block, since these blocks are one-element, containing the constant $u_1$ and $u_2$ maps, respectively.  We take such a bijection for every  $U\subseteq G^*_1\times G^*_2$. Let $\eta$ be the union of these bijections. So $\eta$ is a bijection from $\rel G_1^{\rel G^*_2\times\rel K}$ to $\rel G_2^{\rel G^*_1\times\rel K}$. We finish off the proof by verifying that $\eta$ is an isomorphism.

Let $f,f'\in\rel G_1^{\rel G^*_2\times\rel K}$. By Claim 1,  the pair $(f,f')$ is not an edge of the digraph $\rel G_1^{\rel G^*_2\times\rel K}$ if and only if there exist 
$$(g_1,g_2)\in C(f)\text{ and } (g'_1,g'_2)\in C(f')$$ 
such that none of $(g_1,g'_1)$ and $(g_2,g'_2)$ are edges in $\rel G_1$ and $\rel G_2$, respectively. Since 
$$C(f)=C(\eta(f))\text{ and } C(f')=C(\eta(f')),$$  the latter one is equivalent to the condition that there exist 
$$(g_1,g_2)\in C(\eta(f))\text{ and } (g'_1,g'_2)\in C(\eta(f'))$$ 
such that none of $(g_1,g'_1)$ and $(g_2,g'_2)$ are edges in $\rel G_1$ and $\rel G_2$, respectively. Now by the analogue of Claim 1, this is equivalent that  the pair $(\eta(f),\eta(f'))$ is not an edge of the digraph $\rel G_2^{\rel G^*_1\times\rel K}$. Thus $(f,f')$ is not an edge of the digraph $\rel G_1^{\rel G^*_2\times\rel K}$ if and only if $(\eta(f),\eta(f'))$ is not an edge of the digraph $\rel G_2^{\rel G^*_1\times\rel K}$. So $\eta$ is an isomorphism.
\end{proof}

The following corollary of the preceding theorem is an essential tool in the proof of our main result.

\begin{corollary}
\label{many_isomorphic_powers}
Let $I$ be an arbitrary set. For every $i\in I$, let $\rel G_i$ be a non-complete digraph with a universal vertex $u_i$. Then there exist a digraph $\rel X$ and a non-complete digraph $\rel T$ with a universal vertex such that $\rel G_i^{\rel X}\cong \rel T$  for all $i\in I$.
\end{corollary}
\begin{proof}
Let $\rel N$ be the countably infinite digraph whose edge relation is the equality. Let $\rel X:=\prod_{i\in I}(\rel G^*_i\times\rel K)^{\rel N}$ where for each $i\in I$, $\rel G^*_i$ is the complement of the digraph $\rel G_i-u_i$, and $\rel K$ is an infinite complete digraph whose cardinality is greater than or equal to that of any of the $\rel G_i$. Since for any digraph $\rel G$  $$\rel G\times\rel G^{\rel N}\cong \rel G^{\rel N},$$ for any $i\in I$ $$\rel X\cong(\rel{G}^*_i\times \rel K)\times \rel X.$$ Then by Theorem \ref{two_isomorphic_powers}, for any $i$ and $j$
$$\rel G_i^{\rel X}\cong\rel G_i^{(\rel G^*_j\times \rel K)\times\rel X}\cong\left( \rel G_i^{\rel G^*_j\times \rel K} \right)^\rel X\cong\left(\rel G_j^{\rel G^*_i\times\rel K}\right)^\rel X\cong\rel G_j^{(\rel G^*_i\times \rel K)\times \rel X}\cong\rel G_j^{\rel X}.$$

 Let $i$ be some element of $I$. Since for each $j\in I$, $\rel G_j$ is a non-complete digraph, $\rel G^*_j$ has an edge. Hence $\rel X$ has an edge, too. Therefore, the constant maps from $X$ to $G_i$ induce a subdigraph of $\rel G_i^{\rel X}$ isomorphic to $\rel G_i$.  Hence $\rel G_i^{\rel X}$ is a non-complete digraph. Moreover, the constant $u_i$ map  is a universal vertex in $\rel G_i^{\rel X}$. Then we let $\rel T:=\rel G_i^{\rel X}$ . Clearly, the so defined $\rel X$ and $\rel T$ satisfy the claim.
\end{proof}

Now we have all the tools at our disposal to prove our main result.

\begin{theorem} Congruence permutability is prime.
\end{theorem}
\begin{proof} We are going to prove that the join of two non-permutable varieties $\mathcal{K}$ and $\mathcal{L}$ is non-permutable. As $\mathcal{K}$ and $\mathcal{L}$ are not congruence permutable, by Theorem \ref{Maltsev}, there exist non-symmetric reflexive digraphs $\rel G_0$ and $\rel H_0$ that are compatible digraphs in the varieties $\mathcal{K}$ and $\mathcal{L}$, respectively. Through a series of definitions, starting from $\rel G_0$ and $\rel H_0$, we define some compatible digraphs $\rel G_i$ in $\mathcal K$ and some compatible digraphs $\rel H_i$ in $\mathcal L$. The compatibility of all of these digraphs will be automatic, since their vertex sets and edge relations will be defined from compatible digraphs by existentially quantified (maybe infinite) conjunctions of atomic formulas in the language of digraphs extended by the equality.

We define the digraph $\rel G_1$ on the vertex set  \[
G_1:=\{\ (a,b,c,d)\in G_0^4:\,a\rightarrow b,c,d\ \text{ and }\ b\rightarrow c,d\ \text{ and }\ c\rightarrow d\text{ in }\rel G_0\ \}
\] with the edge relation \[
E(\rel G_1):=\{\ ((a,b,c,d),(a,b',c',d))\in G_1^2:\,b\rightarrow c'\ \text{ and } b'\rightarrow c\text{ in }\rel G_0\ \}.
\] 
Observe that each of the components of $\rel G_1$  has a universal vertex. Indeed, in each component, for some edge $(a,b)$ of $\rel G_0$, $(a,a,b,b)$ is a universal vertex. If $(a,b)$ is an edge but $(b,a)$ is not an edge of $\rel G_0$, then $((a,b,b,b),(a,a,a,b))$ is not an edge in the component of $(a,a,b,b)$ in $\rel G_1$. So at least one of the components of $\rel G_1$ is a non-complete digraph.

Let $\rel R$ be a non-complete component of $\rel G_1$. We define the digraph $\rel G_2$ on the vertex set \[
G_2:=\{\ f\in G_1^{\{0\}\dot\cup R}:\, f(0)\leftrightarrow f(x) \text{ in }\rel G_1 \text{ for all } x\in R \ \}
\] with the edge relation \[
E(\rel G_2):=\{\ (f,f')\in G_2^2:\, f(0)\rightarrow f'(0) \text{ in }\rel G_1  \text{ and } f(x)=f'(x) \text{ for all } x\in R\ \}.
\]

Let $f$ be any vertex of $\rel G_2$ and $f_u$ the vertex obtained from $f$ by changing the value of $f$ at $0$ to the universal vertex $u$ of the component of $f(0)$ in $\rel G_1$. Then $f_u$ is a universal vertex of the component of $f$ in $\rel G_2$.

Moreover, the vertex set
$$\{\ f\in G_2:\ f(x)=u_{\rel R} \text{ for all } x\in R\ \}$$
where $u_{\rel R}$ is a universal vertex of $\rel R$ induces a component isomorphic to $\rel R$ in $\rel G_2$. Thus, $\rel G_2$ contains a non-complete component.
The digraph $\rel G_2$  also contains a complete component. Indeed, the vertex set
$$\{\ f\in G_2:\ f(x)=x \text{ for all } x\in R\ \}$$
gives a component isomorphic to the clique of  the universal vertices of $\rel R$ in $\rel G_2$ .

We define $\rel H_1$ and $\rel H_2$ from $\rel H_0$ by the pattern as $\rel G_1$ and $\rel G_2$ are defined from $\rel G_0$. Thus all components of $\rel H_2$ and $\rel G_2$ contain universal vertices, and there are both complete and non-complete components in each of these digraphs.

By Corollary \ref{many_isomorphic_powers}, there exist a digraph $\rel X$ and a non-complete digraph $\rel T$ with a universal vertex such that for every non-complete component $\rel C$ of the digraphs $\rel G_2$ and $\rel H_2$,  $\rel C^{\rel X}$ is isomorphic to $\rel T$. We define the digraph $\rel G_3$ on the vertex set \[
G_3:=\{\ f\in G_2^X: \text{ there exists } u\in G_2 \text{ such that for all }x\in X,\ f(x)\to u\text{ in } \rel G_2\ \}
\] with the edge relation \[
E(\rel G_3):=\{\ (f,f')\in G_3^2:\,f(x)\rightarrow f'(y)\text{ in } \rel G_2 \text{ for all } x\to y \text{ in }\rel X\ \}.
\] The digraph $\rel G_3$ is the subgraph of $\rel G_2^{\rel X}$ induced by the vertices whose ranges lie in a single component of $\rel G_2$. Thus the components of $\rel G_3$ are the  $\rel X$-th powers of the components of $\rel G_2$. Hence, $\rel G_3$ has only two kinds of components, the ones isomorphic to $\rel T$ and the complete ones. We  define $\rel H_3$ from $\rel H_2$ similarly. The digraph $\rel H_3$ also has only two kinds of components, some isomorphic to $\rel T$ and some complete components.

Now there are some powers of $\rel G_3$ and $\rel H_3$ of the same infinite cardinality, say $\lambda$. Since $\rel G_3$ and $\rel H_3$ are compatible in $\mathcal K$ and $\mathcal L$, respectively, so are these powers. Therefore, the complete digraph $\rel K_{\lambda}$ of cadinality $\lambda$ and the $\lambda$-many element digraph $\rel Q_{\lambda}$  whose edge relation is the equality are compatible digraphs in both varieties  $\mathcal K$ and $\mathcal L$. Hence the digraph $\rel G_3\times \rel K_{\lambda}\times \rel Q_{\lambda}$  is compatible in $\mathcal K$, and the digraph $\rel H_3\times \rel K_{\lambda}\times \rel Q_{\lambda}$ is compatible in $\mathcal L$. Clearly, $\rel K_{\lambda}\times \rel Q_{\lambda}$ is a digraph with $\lambda$-many components where each component is isomorphic to $K_\lambda$. Hence, both of the digraphs $\rel G_3\times \rel K_{\lambda}\times \rel Q_{\lambda}$ and $\rel H_3\times \rel K_{\lambda}\times \rel Q_{\lambda}$ consist of $\lambda$-many complete components of size $\lambda$ and $\lambda$-many non-complete components isomorphic to $\rel T\times \rel K_{\lambda}$. Therefore, the digraphs $\rel G_3\times \rel K_{\lambda}\times \rel Q_{\lambda}$ and $\rel H_3\times \rel K_{\lambda}\times \rel Q_{\lambda}$ are isomorphic and compatible in each of the varieties $\mathcal{K}$, $\mathcal{L}$ and  $\mathcal{K}\vee\mathcal{L}$. Since $\rel G_3\times \rel K_{\lambda}\times \rel Q_{\lambda}$ has a non-complete component that contains a universal vertex, there exist vertices $u,v, w$ in it such that $v\rightarrow u\rightarrow u\rightarrow w$ and $v\not\rightarrow w$. Then, the third condition in Theorem \ref{Maltsev} does not hold for $\mathcal{K}\vee\mathcal{L}$, for otherwise there would be a ternary term $m$ in the language of $\mathcal{K}\vee\mathcal{L}$ such that $v= m(v,u,u)\to m(u,u,w)=w$, a contradiction. Hence, $\mathcal{K}\vee\mathcal{L}$ is not congruence permutable.
\end{proof}

\end{document}